\documentclass[a4paper,12pt,tocn]{amsart}
\usepackage{hyperref}

\usepackage[usenames,dvipsnames,svgnames,table]{xcolor}
\usepackage[all]{xy}
\usepackage{enumitem}
\usepackage[normalem]{ulem}
\newcommand{\tsk}[1]{\textcolor{YellowOrange}}

\newcommand{\togliere}[1]{}
\usepackage{tikz}
\usepackage{tikz-cd}
\usepackage{booktabs}

\makeatletter
\def\@endtheorem{\endtrivlist}
\makeatother

\usepackage[active]{srcltx}
\usepackage{amsmath}
\usepackage{amssymb}
\usepackage{amscd}
\usepackage{amsthm}
\usepackage[latin1]{inputenc}
\usepackage{mathrsfs}  

\usepackage{nicefrac}

\newcommand{\ev}{\operatorname{ev}}

\newtheorem{teo}{Theorem}[section]
\newtheorem*{mainteoA}{Theorem A}
 \newtheorem*{mainteoB}{Theorem B}
\newtheorem{defin}[teo]{Definition}
\newtheorem{prop}[teo]{Proposition}

\newtheorem{lemma}[teo]{Lemma}

\theoremstyle{definition}

\newtheoremstyle{dico}
 {\baselineskip}   
  {\topsep}   
  {}  
  {0pt}       
  {} 
  {.}         
  {5pt plus 1pt minus 1pt} 
  {}          
\theoremstyle{dico}

\numberwithin{equation}{section}

\newcommand{\ra}{\rightarrow}
\newcommand{\C}{\mathbb{C}}
\newcommand{\R}{\mathbb{R}}
\newcommand{\Zeta}{{\mathbb{Z}}}

\newcommand{\meno}{^{-1}}

\newcommand{\alfa}{\alpha}
\newcommand{\alf}{\alpha}

\newcommand{\la}{\lambda}
\newcommand{\restr}[1]          {\vert_{#1}}
\newcommand{\Hom}{\operatorname{Hom}}
\newcommand{\End}{\operatorname{End}}
\newcommand{\Ann}{\operatorname{Ann}}

\newcommand{\liek}{\mathfrak{k}}

\newcommand{\lieg}{\mathfrak{g}}

\newcommand{\liep}{\mathfrak{p}}

\newcommand{\ad}{{\operatorname{ad}}}

\renewcommand{\setminus}{-}

\newcommand{\om}{\omega}

\renewcommand{\phi}{\varphi}

\newcommand{\cd}{\cdot}
\newcommand{\im}{\operatorname{Im}}

\newcommand{\lra}{\longrightarrow}

\newcommand{\id}{\operatorname{id}}
\newcommand{\est} {\Lambda}

\renewcommand{\sp}{\mathfrak{sp}}

\newcommand{\PP}{\mathbb{P}}   
   
\renewcommand{\phi}             {\varphi}

\newcommand{\sieg}{\mathfrak{S}}

\newcommand{\M}{\mathsf{M}}

\newcommand{\A}{\mathsf{A}}

\newcommand{\Ag}{\mathsf{A}_g}

\newcommand{\Sp}                {\operatorname {Sp}}

\renewcommand{\Im}              {\operatorname{Im}}

\newcommand{\debar }            {\bar {\partial } }

\newcommand{\ag}{\mathsf{A}_g}
\newcommand{\mg}{\mathsf{M}_g}

\newcommand{\mihi}[1]{}

\newcommand{\vc} {{V_\C}}
\newcommand{\vcs} {{V_\C^*}}
\newcommand{\ove}[1]{\overline{#1}}
\newcommand{\bra}{\mathbf{B}}
\newcommand{\brad}{\bra^*}
\newcommand{\barc}{ {\bar {C}}}

\def\LaTeX{%
	\let\Begin\begin
	\let\salta\relax
	\let\finqui\relax
	\let\futuro\relax}
      \LaTeX

\def\Edef{\end{definition}}

\def\Elem{\end{lemma}}

\def\Edim{\end{proof}}

\def\Eprop{\end{prop}}

\def\Ecor{\end{corollary}}

\def\Ethm{\end{theorem}}

\def\Erem{\end{remark}}

\def\vc{V_{\mathbb{C}}}
\def\delbar{\bar{\partial }}
\def\del{\partial}

\def\Berg{\mathbf{K}}
\def\ku{k_u}
\def\kv{k_v}

\begin{document}
\author{Alessandro Ghigi, Carolina Tamborini}
\title{Bergman kernel and period map for curves}

\address{Universit\`{a} di Pavia}
\email{alessandro.ghigi@unipv.it}
\email{c.tamborini2@campus.unimib.it} 

\thanks{The authors were partially supported by MIUR PRIN 2017
  ``Moduli spaces and Lie Theory'' , by MIUR, Programma Dipartimenti
  di Eccellenza (2018-2022) - Dipartimento di Matematica
  ``F. Casorati'', Universit\`a degli Studi di Pavia and by INdAM
  (GNSAGA).  }

\maketitle
\begin{abstract}
  As for any symmetric space the tangent space to Siegel upper-half
  space is endowed with an operation coming from the Lie bracket on
  the Lie algebra.  We consider the pull-back of this operation to the
  moduli space of curves via the Torelli map.  We characterize it in
  terms of the geometry of the curve, using the Bergman kernel form
  associated to the curve.  It is known that the second fundamental
  form of the Torelli map outside the hyperelliptic locus can be seen
  as the multiplication by a certain meromorphic form.  Our second
  result says that the Bergman kernel form is the harmonic
  representative - in a suitable sense - of this meromorphic form.

\end{abstract}

\tableofcontents

\section{Introduction}

\subsection{} \label{intro} Let $X$ be a Riemannian symmetric
space. For a fixed point $x \in X$ we have $X=G/K$, where $G$ is a Lie
group (independent of $x$) and $K=G_x$ is the stabilizer of
$x$. Moreover there is a Cartan decomposition
$\lieg = \liek \oplus \liep$ such that $[\liep, \liep ] \subset \liek$
and $[\liek, \liep ]\subset \liep$.  Since $\liep \cong T_xX$, the Lie
bracket on $\lieg$ gives rise to a kind of operation
$B_x : T_xX \times T_x X \cong \liep \times \liep \ra \liek =
\lieg_x$.  Let $\mathbf{S} = G \times_K \liek$ be the homogeneous
bundle over $X$ corresponding to the adjoint representation of
$K$. Then $\mathbf{S}_x = \lieg_x$ for any $x$, so $B$ is a section of
$\est^2 T^*X \otimes \mathbf{S}$.  Since the differential geometry of
$X$ can be studied by means of Lie theory, the tensor $B$, which
reflects Lie bracket, is of central importance.

\subsection{} \label{intro2}
Since the tensor $B$ is invariant by the action of $G$ it makes sense
also on any locally symmetric space $X$.  In this paper we consider
the case where $X$ is $\Ag$, the moduli space of principally polarized
abelian varieties of dimension $g$ over $\mathbb{C}$, which is a
locally symmetric space obtained as a quotient of the Siegel upper
half-space $ \sieg_g$.  Denote by $\M_g$ the moduli space of curves of
genus $g$. We are interested in the Torelli map $j : \M_g \ra \A_g$,
which associates to $[C]\in \M_g$ its Jacobian variety $[JC]\in \A_g$.
Our motivation comes from the study of totally geodesic subvarieties
of $\Ag$ that are generically contained in $j(\mg)$
\cite{hain,cfg,deba}. This is also connected to the Coleman-Oort
conjecture \cite{moonen-oort}.  The tensor $B$ controls the local
geometry of $\Ag$ and its pull-back $\bra=j^*B$ to $\mg$ should give
important information on the extrinsic geometry of the inclusion
$j(\mg) \subset \ag$. For example we expect that the study of $\bra$
will give constraints on the existence of Lie triples tangent to
$\mg$.

\subsection{}
The first step in this direction is the computation of $\bra$ at a
moduli point $[C]\in \mg$ in terms of the geometry of the curve
$C$. This is the first main result of this note.  Let $\barc$ denote
the conjugate curve, i.e. with the opposite complex structure.  We
first show that the dual map of $\bra$ can be seen as map
$$\bra^*: H^0(K_C)\otimes H^0(K_{\overline{C}})\rightarrow
H^0(2K_C)\otimes H^0(2K_{\overline{C}}).$$ Secondly, we consider the
algebraic surface $Z=C\times \overline{C}$. By K\"unneth formula
$H^0(Z, K_{Z})\simeq H^0(K_C)\otimes H^0(K_{\overline{C}})$ and
$H^0(2K_{Z})\simeq H^0(2K_C)\otimes H^0(2K_{\overline{C}})$.  With
these identifications we prove the following.
\begin{mainteoA}
  The map
\begin{gather*}
  \bra^* : H^0(Z,K_Z) \lra H^0(Z, 2K_Z)
\end{gather*}
coincides with the multiplication by $-i\Berg$,  where  $\Berg\in H^0(Z, K_{Z})$  is the Bergman kernel of the curve $C$.
\end{mainteoA}
(See Theorem \ref{teomulti}.
See \ref {ssBergman} for the definition of Bergman kernel in the
sense we need.)  In other words the Bergman kernel $\Berg$ governs the
restriction of the Lie bracket to $dj(T\M_g)$.

\subsection{}Another approach to the extrinsic geometry of inclusion
$j(\M_g) \subset \A_g$ outside the hyperelliptic locus uses the second
fundamental form. If $C$ is non-hyperelliptic, the second fundamental
form at $[C]$ has been interpreted  as the 
multiplication 
by a holomorphic section
$\hat{\eta}$ of the line bundle $K_S(2\Delta)$, where $S=C\times C$ and
$\Delta\subset S$ is the diagonal (see  \cite{cfg,cpt}).  This leads back the study of the
behavior of the second fundamental form to the study of the $2-$form
$\hat{\eta}\in H^0(S, K_S(2\Delta))$.

\subsection{} The form $\hat{\eta}$ has been further studied in
\cite{bcfp} in relation with projective structures on compact Riemann
surfaces. Section $5$ of \cite{bcfp} is dedicated to the study of the
cohomology class of the form $\hat{\eta}$ and contains a
characterization of $\hat{\eta}$ as the unique 
element  (up to multiples)  of $H^0(S, K_S(2\Delta))$ with cohomology class in
$H^2(S\setminus \Delta)$ of pure type $(1,1)$.

\subsection{}
In our second result we  give an explicit description of the harmonic representative of the cohomology class of $\hat{\eta}$:
%
%
%
\begin{mainteoB}
  The Bergman kernel is the $(1,1)-$harmonic representative of the
  cohomology class of $\hat{\eta}\in H^0(S, K_S(2\Delta))$ in
  $H^2(S\setminus \Delta, \mathbb{Z})$.  More precisely, there exists
  $\alpha\in H^0(S, \mathcal{A}^{1,0}(\Delta))$ such that
  $$\hat{\eta}- 2\pi \Berg =d\alpha, $$
  that is $\del\alpha=\hat{\eta}$ and $\delbar \alpha=-2 \pi \Berg$.	
\end{mainteoB}
(See Theorem \ref{teoeta}.)  It is quite hard to control the behaviour
of $\hat{\eta}$ outside of the diagonal. Only along $\Delta$ its
behaviour admits an algebraic description, via the second Gaussian map
$\mu_2$, see \cite{cpt}. We expect the above result to allow some
better understanding of $\hat{\eta}$ and the second fundamental form.

\medskip

{\bfseries \noindent{Acknowledgements}}.  The idea to study the Lie
bracket restricted to moduli space arose from discussions with Gian
Pietro Pirola.  The idea of using Bergman kernel in the study of the
map $\bra^*$ was suggested by Indranil Biswas.  We heartily thank both
of them. We also thank Paola Frediani for several interesting
discussions related to the subject of this paper.

\section{Bergman kernel}
\subsection{}
\label{ssBergman}
Let $C$ be a smooth complex projective curve of genus $g\geq 1$. Set
$S:=C\times C$ and let $p,q: S\rightarrow C$ be the projections
$p(x,y)=x$, $q(x,y)=y$.

Let $\barc$ denote the conjugate variety and set $Z:=C\times
\barc$. $Z$ coincides with $S$ as a real manifold, but has a different
complex structure. The projections $p : Z \ra C$, $q: Z \ra \barc$ are
holomorphic.

Denote by $h$ the Hodge Hermitian product on $H^0(C, K_C)$, defined by
\begin{gather*}
  h(\alpha, \beta)=:i\int\alpha\wedge \overline{\beta}.
\end{gather*}
\begin{defin}
  Let $\omega_1,..., \omega_g$ be a unitary basis for $H^0(C,
  K_C)$. Then
	$$\Berg:=\sum_{j=1}^{g} p^*\omega_j \wedge q^*\overline{\omega}_j$$
	is a well-defined $(1,1)-$form on $S$ independent of the
        choice of the unitary basis. It is called the \textbf{Bergman
          kernel form} of the algebraic curve $C$.
      \end{defin}
      This is the  definition of the Bergman kernel form on an arbitrary complex manifold due to Kobayashi \cite{kobayashi-domains}.
It generalizes the classical Bergman kernel on open domains in $\C^n$.
      $\Berg$ can also
      be seen as a holomorphic 2-form on $Z$.  In particular it is a
      harmonic form with respect to any K\"ahler metric on $Z$.  If we
      consider it as a (1,1)-form on $S$, it is harmonic for any
      K\"ahler metric on $S$ which is K\"ahler also on $Z$. In
      particular it is harmonic for any product metric.

      If $x,y\in C$, $T_{(x,y)}S = T_xC \oplus T_yC$.  Thus elements
      of $T_{(x,y)} S$ are pairs $(u,v)$ with $u\in T_xC$ and
      $v\in T_y C$.  Since $\Berg$ is a (1,1)-form, its behaviour is
      controlled by the values $\Berg((u,0),(0,\overline{v}))$ for
      $u\in T_x^{1,0}C$, $v\in T_y^{1,0}C$.

\subsection{}
Although not needed in the following, it is interesting to point out
the following relation between the Bergman kernel and the period
matrix associated to the algebraic curve $C$
(cf. \cite[eq. (2.4)]{wentworth_asymptotics_1991}). Let $Q$ denote the
intersection form on $H_1(C, \mathbb{Z})$ and let $\{a_i, b_i\}$ be a
symplectic basis for $(H_1(C, \mathbb{Z}), Q)$. Consider a basis
$\omega_1,...,\omega_g$ of $H^0(C, K_C)$ normalized with respect to
$\{a_i, b_i\}$ and the period matrix $Z=(z_{ij})$ with
$z_{ij}=\int_{b_j}\omega_i$.  Then, with respect this basis, the
Bergman kernel has the form
\begin{gather}
  \label{berg-periodi}
  \Berg=\frac{1}{2} \sum_{i,j} (\Im Z )^{ij} \, p^*\omega_j \wedge
  q^*\overline{\omega}_j
\end{gather}
where $(\Im Z)^{ij}$ denote the coefficients of
  $(\Im Z)\meno$.

  To check this observe first that
  $h(\omega_i, \omega_j)=2 \Im z_{ij}$.  Indeed let
  $\mathscr{B}=\{a_i^*, b_j^*\}$ be the dual basis. If
  $D: H^1(C) \ra H_1(C)$ is Poincar\'e duality, then $D a_i^* = b_i $
  and $Db_i^* = -a_i$, so $\mathscr{B}$ is symplectic for
  $Q^*(\alf,\beta) = \int_C \alf \cup \beta$. Since
  $ \om_i = a^*_i + \sum_{k=1}^n z_{ik} b^*_k$, the result follows.

  Now \eqref{berg-periodi} is a consequence of the following general
  fact: if $\alpha_1,..., \alpha_g$ is a basis of $H^0(C, K_C)$ and
  $A$ is the matrix with entries $a_{ij}:=h(\alf_i, \alf_j)$, then
  $ \Berg=\sum_{i,j} a^{ij} p^*\alpha_i \wedge q^*\overline{\alpha}_j $.

\subsection{}
Next we show how to recover the Bergman kernel using the so-called
elementary potentials.  Let $(U,z)$ be a chart centered at $x\in C$
and set $u=\frac{\partial}{\partial z}(x).$ Classical results ensure
the existence of a harmonic function $f_u\in C^{\infty}(C -\{x\})$
such that $f_u=- \frac{1}{z}+g(z)$ on $U-\{x\}$ for some
$g\in C^{\infty}(U)$. The function $f_u$ is unique up to an additive
constant and it is called \emph{elementary potential} (see \cite[\S
3]{cfg}). We recall some of its properties that will be relevant in
our analysis.  It follows from the definition that $\delbar f_u$ is
smooth on $C$ and that
\begin{gather}
  \label{debarfu}
  \int_C\omega \wedge (-\delbar f_u)=2\pi i \omega(u),  \quad \text{for } \omega\in
  H^0(C, K_C).
\end{gather}
(See \cite[Section 3]{cfg} for more details.)  This shows that
elementary potentials are related to evaluation and the canonical
map. Therefore they are clearly related to Bergman kernel as we show
now.

\subsection{}For $x\in C$ and $u\in T_xC$, let
$\ev_u: H^0(C, K_C)\rightarrow \mathbb{C}$ be the evaluation map and
let $\ku\in H^0(K_C)$ be such that
\begin{gather}
  \label{defku}
  \ev_u=h(\cdot , \ku).
\end{gather}

\begin{lemma}
\label{berg}
  For $u\in T^{1,0}_xC$, $v\in T^{1,0}_y C$, with $x,y\in C$, we have
  \begin{gather*}
    \Berg((u,0),(0,\overline{v}))=h(\kv,\ku) = k_v(u).
  \end{gather*}
\end{lemma} 
\begin{proof} Let $\omega_1,..., \omega_g$ be a unitary basis for $H^0(K_C)$. Then $\kv=\sum_j\lambda_j \omega_j,$ with $\lambda_j= h(\kv, \omega_j)=\overline{\omega_j(v)}$. Thus 
	$h(k_v, k_u)=k_v(u)=\sum_{j} \overline{\omega_j(v)}\omega_j(u)=\Berg((u,0),(0,\overline{v})).$
\end{proof}

\begin{lemma} \label{berg-potelem} Let $x,y\in C$, and
  $u\in T^{1,0}_xC$, $v\in T^{1,0}_y C$.  If $f_u$ is an elementary
  potential, then $ \delbar f_u=2\pi \overline{k_u}$. In particular
\begin{gather*}
  \Berg((u,0), (0, \overline{v}))=\frac{1}{2\pi}
  \delbar f_u(\overline{v}).
\end{gather*} 
 \end{lemma}

\begin{proof}
  From \eqref{debarfu} and \eqref{defku} it follows that
  $\int_C\omega \wedge (\delbar f_u)=-2\pi i \omega(u)=-2\pi i\
  h(\omega, k_u)=2\pi \int_C \omega \wedge\ \overline{k_u}$. So
  $\debar f_u$ and $2\pi \ove{k_u}$ have the same cohomology
  class. Since both are harmonic they coincide. Next by the previous
  Lemma
  $ \Berg((u,0), (0, \overline{v}))= h(k_v, k_u)= \overline{h(k_u,
    k_v)}=\overline{k_u(v)}=\overline{k_u}(\overline{v})=\frac{1}{2\pi}
  \delbar f_u(\overline{v})$.
\end{proof}

\section{Lie bracket}

In this section we study the Lie bracket i.e. the tensor $B$ on Siegel
space, introduced in \ref{intro}, and prove Theorem A.  We start by
recalling something about Siegel upper half-space.  Next we go through
several identications of the tangent space to $\sieg_g$ and write down
$B$ in terms of them (Proposition \ref{lemdue}).  Given a curve, we
apply this to its Jacobian, i.e. we consider $\bra$ as in \ref{intro2}  (Proposition \ref{teojac}).
We recall further 
identifications using the conjugate curve and the surface
$Z=C\times \barc$.  This allows to understand the dual map  $\bra^*$   as a map
$H^0(Z,K_Z) \ra H^0(Z,2K_Z)$.  Finally using the elementary potentials
we prove our main result Theorem \ref{teomulti}.

Let $(V,Q)$ be a real symplectic vector space.
If $ J \in \End V$ satisfies $ J^2=-I_V$ and $ J^*Q=Q$, the bilinear
form
$
  g_J(v, v') := Q(v, J v' )
$
is symmetric.  Siegel upper half-space is defined as 
\begin{gather*}
\sieg:=  \sieg(V,Q):=\{ J \in \End V: J^2=-I_V, J^*Q=Q, g_J 
  \gg 0\}.
\end{gather*}
For every $J$ we denote $V_{-1,0}(J)$ and $V_{0,-1}(J)$ the $\pm i$-eigenspaces of $J$ on $V_\C$. We also set 
\begin{gather*}
H^{1,0}_J :=\Ann V_{0,-1} \quad  H^{0,1}_J:=\Ann V_{-1,0}(J).  
\end{gather*}
We usually drop $J$ in the notation.
When $V=\R^{2g}$ and $Q$ is the standard form, we write $\sieg_g$.
The symplectic group $\Sp:=\Sp(V,Q)$ acts on $\sieg $ by
conjugation. This action is transitive and $\sieg$ is a Hermitian
symmetric space.
For $X \in \End V$ set $ Q_X := Q(\cd, X\cd ) $. Then
$ \sp=\sp(V,Q) = \{ X \in \End V : Q_X  $ is symmetric$\}$. If $\sp = \sp_J \oplus \liep$ is the Cartan decomposition at $J\in \sieg$, then
\begin{gather*}
  \liep = \{ X \in \sp: XJ+JX = 0\},
  \quad \sp_J = \{ X \in \sp: [J,X]=0\}.
\end{gather*}
We endow $\liep \cong T_J\sieg$ with the complex structure
$\hat{I}:= (1/2)\ad J$. Then
\begin{gather}
  \notag
  \liep_\C  
  = \{ X \in \sp_\C:  X(V_{-1,0} ) \subset V_{0,-1} \text { and } X(V_{0,-1} )\subset V_{-1,0}  \}
  \\
  \label{p10} \liep^{1,0} = \{X\in \Hom(V_{0,-1}, V_{-1,0}): Q_X
  \text{ is symmetric}\},
\end{gather}

\subsection{}
  \label{lagrangiani}
  We have an isomorphism
  \begin{gather*}
    \phi_Q: V_\C \lra V_\C^* , \quad \phi_Q( v):= Q(\cd, v).
  \end{gather*}
  Its inverse is denoted by $\psi_Q:=\phi_Q\meno$.
  For any Lagrangian subspace $L \subset V_\C$ the isomorphism
  $\phi_Q$
  maps $L$ onto $\Ann(L)$. Therefore $\phi_Q$ gives an isomorphism
  $V_{0,-1} \cong H^{1,0}$.

  As mentioned in \ref{intro} we are interested in the Lie bracket
  which can be seen as a section of a bundle over the symmetric space
  $\sieg$.  We wish to compute
  $B_J \in \est^2 T_J^*\sieg \otimes \sp_J$.  As usual it is useful to
  look at $B$ thourgh its complexification
\begin{gather*}
  B_J : ( T_J\sieg)_\C \times ( T_J\sieg)_\C \lra (\sp_J)_\C.
\end{gather*}
Recall that $  ( T_J\sieg)_\C =\liep_\C= \liep^{1,0}\oplus \liep^{0,1}$,
where $\liep^{1,0} $ is given by \eqref{p10}
and 
\begin{gather}
  \label{spjc} \begin{aligned} (\sp_J)_\C= & \{X \in \End V_\C:
    X(V_{-1,0} ) \subset V_{-1,0}, \\
    & X(V_{0,-1} ) \subset V_{0,-1},
    Q_X \text{ is simmetric}\}.\\
  \end{aligned}
\end{gather}
\begin{lemma}
  The map $B_J$ is of type $(1,1)$, i.e. it vanishes on vectors of the
  same type.
\end{lemma}
\begin{proof}
Since $B_J$ is real, it is enough to show
that it vanishes on pairs of $(1,0)$-vectors.
If $X, Y \in \liep^{1,0}$, then $X ( \vc ) \subset V_{-1,0}$ and $Y\restr{V_{-1,0}} = 0$, thus $YX =0$. For the same reason also $XY=0$. Thus $B_J(X,Y) = XY -YX =0$.
\end{proof}

If $X \in \End V_\C$, let $X^* \in \End V_\C^*$ denote the transpose.
The transposition map $X \mapsto X^*$ is a canonical isomorphism
$\End V_\C \cong \End V_\C^*$.  It is useful to reinterpret everything
in terms of $\End \vcs$ rather than $\End \vc$.  Set
\begin{gather*}
   Q^*:=\psi_Q^* Q .
\end{gather*}
(Notation as in \ref{lagrangiani}.) Then $Q^*$ is a symplectic form on
$V^*_\C$.
\begin{lemma}\label{lemzero} If $X\in \End \vc$, then $Q_X$ is symmetric iff
  $Q^*_{X^*}$ is symmetric.
\end{lemma}
\begin{proof}
  Define $\tilde{X} \in \End V_\C$ by $Q(Xa,b) =Q(a,\tilde{X}b)$. Then
  \begin{gather*}
    X^*\phi_Q a = X^* Q(\cd, a) = Q (X\cd ,a) = Q (\cd, \tilde{X} a) =
    \phi_Q \tilde{X}a .
  \end{gather*}
  Given $\alfa, \beta \in V^*_\C$ let $a=\psi_Q\alf, b=\psi_Q
  \beta$. Then
  \begin{gather*}
    -Q^*_{X^*}(\alf,\beta)=  Q^*(X^*\alf, \beta) = Q(\psi_Q X^*\phi_Qa , \psi_Q \phi_Q b) = \\
    =Q(\tilde {X}a , b ) = Q(a, Xb) = Q_X(a,b).
  \end{gather*}
  So $-Q^*_{X^*}(\alf,\beta)= Q_X(a,b)$.  The statement follows.
\end{proof}

\begin{lemma}\label{lemuno}
 If
  $X\in \End V_\C$, then
  \begin{align*}
    X\restr{V_{-1,0}} =0 & \Longrightarrow \im X^* \subset H^{0,1} \\
    \im X \subset V_{-1,0} & \Longrightarrow  X^*\restr{H^{0,1}} = 0.
  \end{align*}
\end{lemma}
\begin{proof}
  Recall that for any linear map $L : E \ra F$ of vector spaces 
  $\Ann \im L = \ker L^*$.  Now
  $ X\restr{V_{-1,0}} =0 \Longrightarrow V_{-1,0} \subset \ker X
  \Longrightarrow H^{0,1} = \Ann V_{-1,0} \supset \Ann \ker X = \im
  X^*$.  And
  $ \im X \subset V_{-1,0} \Longrightarrow \ker X^* = \Ann \im X
  \supset \Ann V_{-1,0}=H^{0,1}$.
\end{proof}

\begin{prop}\label{lemdue}
  There are canonical isomorphisms
  \begin{gather}
    \label{liep10} \liep^{1,0} \cong \{ t \in \Hom (H^{1,0},H^{0,1}) :
    Q^*_{t} \text
    { is symmetric}\}, \\
    \notag (\sp_J)_\C \cong \End H^{1,0}.
  \end{gather}
  Using these isomorphism $B_J$ gets identified with the map
  \begin{gather*}
    B_J: \liep^{1,0} \times \ove{\liep^{1,0}} \lra \End H^{1,0} \quad
    (s,\bar{t}) \mapsto \bar{t}{s}.
  \end{gather*}
\end{prop}
\begin{proof}
  The first isomorphism is simply the restriction to $\liep^{1,0}$ of
  the map $X\mapsto X^*$. Lemmata \ref{lemzero} and \ref{lemuno} show
  that indeed the image of $\liep^{1,0}$ is the set of
  $X^*\in \End\vcs$ that vanish on $H^{0,1}$, have image in $H^{0,1}$
  and such that $Q^*_{X^*}$ is symmetric.  To describe the second
  isomorphism start from \eqref{spjc}. Again the Lemmata show that the
  map $X\mapsto X^*$ sends $(\sp_J)_\C$ to the set of
  $X^* \in \End\vcs$ that preserve each $H^{p,q}$ and such that
  $ Q^*_{X^*} $ is symmetric.  The latter means that
  $Q^* (u, X^* v) = Q^*(v, X^*u)$. This identity is trivial if $u $
  and $v$ have the same type, since in that case both term vanish.
  Hence $X^* \restr{H^{1,0}}$ is an arbitrary endomorphism of
  $H^{1,0}$.  On the contrary the identity shows that for any
  $v\in H^{0,1}$ the value $X^*v$ is determined by
  $X^* \restr{H^{1,0}}$.  Hence
  \begin{gather}
    \label{spjc2}
    (\sp_J)_\C \lra \End H^{1,0}, \quad X \mapsto X^*\restr{H^{1,0}}
  \end{gather}
  is the desired isomorphism.  Now let $X,Y \in \liep^{1,0}$ and set
  $s:=X^*$, $t:=Y^*$. Then $ B(X,\ove{Y})^* = \bar{t} s - s
  \bar{t}$. Since $\bar{t}\restr{H^{1,0}} =0$, in the isomorphism
  \eqref{spjc2} $B(X,Y) \in (\sp_J)_\C$ corresponds to
  $ B(X,\ove{Y})^*\restr{H^{1,0}} = \bar{t} s$.
\end{proof}



Since $H^{1,0} = \Ann V_{0,-1} $, we have
$\Ann H^{1,0} = V_{0,-1}$. So there is a canonical isomorphism
$(H^{1,0})^* \cong \vc / \Ann H^{1,0} = \vc /V_{0,-1} \cong V_{-1,0}$.
We treat this isomorphism as an identity.  By \ref{lagrangiani}
$\phi_Q$ maps $V_{-1,0} $ isomorphically onto $H^{0,1}$.  Thus
$\psi_Q=\phi_Q\meno$ restricts to an isomorphism
\begin{gather}
  \label{1001}
  \psi_Q: H^{0,1}\stackrel{ \cong}{\lra} (H^{1,0})^* .
\end{gather}
\begin{lemma}
  \label{inversi}
For $\bar{\om} \in H^{0,1}$, we have $\psi_Q(\bar{\om}) = Q^*(\bar{\om} , \cd)$.
\end{lemma}
\begin{proof}
  First we claim that $\phi_{Q^*}\phi_Q = -\id_{\vc}$ i.e.
  $\psi_Q= - \phi_{Q^*}$.  Indeed fix $v\in \vc$.  That
  $\phi_{Q^*}\phi_Q (v) = - v$ means that $Q^*(\cd, \phi_Q(v)) = -v$,
  i.e. that $ Q^*(\la, \phi_Q(v)) = - \la (v) $ for any $\la\in \vcs$.
  Assume $\la = \phi_Q(w)$ for $w\in \vc$. Then $\la (v) = Q(v,w)$ and
  $Q^*(\la, \phi_Q(v)) $ $= Q^*(\phi_Q(w) , \phi_Q(v)) = Q(w,v)$.
\end{proof}

\togliere{
\begin{lemma}
  \label{simmetria}
  A linear map $t : W \ra W^*$ lies in $S^2W^*$ if and only if
  $t(w)(w') = t(w')(w)$ for any $w, w'\in W$.
\end{lemma}
Now consider $ \liep^{1,0}$ as described in \eqref{liep10} and
consider the map that sends $ t \in $ $ \Hom (H^{1,0},H^{0,1}) $ to
$ \psi_Qt \in \Hom(H^{1,0}, (H^{1,0})^*) = (H^{1,0})^* \otimes
(H^{1,0})^* $.  This map is injective.  We claim that it maps
$\liep^{1,0}$ onto $S^2 H^{1,0\, *}$.  It is enough to prove that
$Q^*_t$ is symmetric iff $\psi_Q t \in S^2(H^{1,0})^*$. But
$\psi_Q = - \phi_{Q^*}$, so
\begin{gather*}
  \psi_Q t (\alf)(\alf') = - \phi_{Q^*} (t\alf) (\alf') = - Q^*(\alf',
  t\alf) = - Q^*_t (\alf', \alf).
\end{gather*}
Applying Lemma \ref {simmetria} we get in fact that $Q^*_t$ is
symmetric iff $t \in S^2H^{1,0 \, *}$. Thus we have an isomorphism
\begin{gather}
  \label{liepsim}
  \liep^{1,0} \stackrel{\cong}{\lra} S^2 H^{1,0\, *} , \quad t\mapsto
  \psi_Q t.
\end{gather}
The inverse isomorphism is the following: take
$\tau \in S^2H^{1,0\, *} $; this can be seen as a linear map
$\tau : H^{1,0} \ra H^{1,0 \, *}$. Then $\tau$ is mapped to
$\phi_Q\tau \in \liep^{1,0} \subset \Hom (H^{1,0},H^{),1})$.
}


Now consider the period map $j : \mg \lra \ag$.  Let $x\in \mg$ be the
moduli point of a curve $C$: $x=[C]$.  If we fix a symplectic basis of
$H_1(C,\Zeta)$ we get a sympletic isomorphism of $H_1(C, \R)$ with the
intersection form onto $ (\R^{2g}, Q)$. Thus the Hodge decomposition
$H^1(C,\C) = H^1(C) \oplus H^{0,1}(C) $ gives a complex structure on
$H^1(C,\C)$, hence a point $J\in \sieg_g$.
\togliere{In a completely equivalent
way (at least for what matters in this paper) one could also lift $j$
to Teichm\"uller space to get a period map with values directly in
$\sieg_g$.}

In the following we use $T_x\mg $ to denote the \emph{real} tangent
space i.e. the tangent space of $\mg$ as a differentiable manifold
(and similarly for $\ag)$. Thus
$(T_x\mg) _\C =T^{1,0}_x \mg \oplus T^{0,1}_x\mg$ and
$T_x^{1,0}\mg = H^1(C,T_C)$, while
$(T_x\ag)_\C = (T_J \sieg_g)_\C = \liep^{1,0}\oplus \liep^{0,1}$.  By
a theorem of Griffiths the map $ dj_x : H^1(C,T_C) \lra \liep^{1,0} $
using the interpretation \eqref{liep10} is given by
\begin{gather*}
  dj_x (\xi) = \xi \cup \cd : H^{1,0} \lra H^{0,1} , \quad dj_x (\xi)
  (\om) = \xi \cup \om .
\end{gather*}
(See e.g. \cite[pp. 234ff]{voisin-theorie}.)  Now
$T_x ^{0,1} \mg = \ove { H^1(C,T_C)} $ is the conjugate vector space,
i.e. it has the same underlying real vector space as $H^1(C,T_C)$ but
multiplication by $i$ is replaced with multiplication by $-i$.  Since
$j$ is holomorphic, its differential is a direct sum of the map
$dj_x: H^{1}(C,T_C) \ra \liep^{1,0}$ and its conjugate. Hence for
$\bar{\eta} \in \ove{H^1(C,T_C)}$ and $\om\in H^{1,0}(C)$ we have
\begin{gather*}
  dj_x (\bar{\eta}) = \bar{\eta}\cup  \cd  : H^{1,0} \lra H^{0,1} , \quad dj_x (\bar{\eta} )
  (\bar{\om}) = \ove{\eta \cup \om} .
\end{gather*}
As mentioned in the Introduction
our goal is to study the  map
\begin{gather*}
  \bra_x :=dj_x ^* B_J : H^1(C,T_C) \times \ove{H^1(C, T_{C})}
  \lra \End H^{1,0}(C).
\end{gather*}
The following is a consequence of  Proposition \ref{lemdue}.
\begin{prop}
  \label{teojac}
  For $\xi, \eta \in H^1(C,T_C)$ and $\om \in H^{1,0}(C)$, we have
  \begin{gather*}
    \bra (\xi, \bar{\eta}) (\om) = \bar{\eta} \cup ( \xi \cup \om).
  \end{gather*}
\end{prop}

Once again it is useful to dualize. This time we dualize the map
$\bra$ itself.  Using \eqref{1001} we can describe the domain of
$\brad$ as follows:
\begin{gather*}
  ( \End H^{1,0})^* = (H^{1,0\, *} \otimes H^{1,0})^* = H^{1,0}
  \otimes H^{1,0\, *} \cong H^{1,0}\otimes H^{0,1} .
\end{gather*}
More explicitely, let $\om, \om \in H^{1,0}$ and $t\in \End
H^{1,0}$. Then $\om \otimes \bar{\om}' \in H^{1,0}\otimes H^{0,1}$.
Recalling Lemma \ref{inversi} one easily verifies that the
corresponding element of $(\End H^{1,0})^*$ is the linear functional
mapping $t$ to $Q^*(\bar{\om}', t {\om} )$.

The dual of $H^1(C,T_C)$ is $H^0(C,2K_C)$. Thus the dual of $\bra$ is
defined on $ H^0(C,2K_C) \otimes \ove{H^0({C}, 2K_{C} )}$ and maps to
$H^{1,0}(C)\otimes \ove{H^{1,0}(C)}$.

Denoting by $\barc $
the conjugate variety we have
\begin{gather*}
  \ove{H^0({C}, 2K_{C} )} = H^0(\barc, 2K_{\barc}), \quad \ove{H^{1,0}(C)} = H^{1,0}(\barc).
\end{gather*}
Thus $\bra^*$ is a map from $ H^{1,0}(C) \otimes H^{1,0}(\barc) $ to
$ H^0(C,2K_C)\otimes H^0(\barc, 2K_\barc)$.  We further reinterpret
domain and target of $\brad$ as spaces of sections of appropriate
bundles on $Z=C\times \barc$.  Denoting by $p:Z\ra C$ and
$q: Z \ra \barc$ the projections and given bundles $L\ra C$ and
$M\ra \barc$, set $ L\boxtimes M : = p^*L \otimes q^* M$.  The map
\begin{gather*}
  H^0(C,L) \otimes H^0(\barc,M) \lra H^0(Z, L \boxtimes M), \quad s\otimes
  t \mapsto p^*s \otimes q^* t,
\end{gather*}
is an isomorphism.  For any positive integer $n$ there is a canonical
isomorphism
\begin{gather}
  \label{isomKCKS}
  K_Z ^n\cong K_C^n \boxtimes K^n_\barc ,
\end{gather}
obtained as follows: if $\alf\in K_{C,x}$ and $\beta \in K_{\barc,y}$,
denote by $\alfa^n \in (K_{C,x})^{\otimes n}$ and
$\beta ^n \in (K_{\barc,y})^{\otimes n}$ the tensor powers. Then
$ \alfa^n\otimes \beta^n \in
(K_{C}^n\boxtimes K^n_\barc) _{(x,y) } $, while
$ (p^*\alfa \wedge q^* \beta)^n \in K_{Z,(x,y) } $.
The isomorphism \eqref{isomKCKS} maps $\alfa^n\otimes \beta^n$ to
$ (p^*\alfa \wedge q^* \beta)^n$.

We now prove Theorem A.
\begin{teo}
  \label{teomulti}
  The map
\begin{gather*}
  \bra^* : H^0(Z,K_Z) \lra H^0(Z, 2K_Z)
\end{gather*}
coincides with the multiplication by $-i\Berg$.
\end{teo}
\begin{proof}
  Fix a point $x\in C$ and a chart $(U,z)$ centered in $x$. Set
  $u=\frac{\partial}{\partial z}(x)$ and consider the Schiffer
  variation $\xi_u$ at $x\in C$.  We recall that
  \begin{gather}
    \label{dj}
    \xi_u \cup = - 2\pi \ev_u \otimes \ove{k}_u. 
  \end{gather}
  Indeed fix a Dolbeault representative
  $\phi=\frac{\delbar b}{z}\frac{\del}{\del z},$ where
  $b\in C^{\infty}$ is a bump function which is equal to 1 in a
  neighbourhood of $x$. For $\omega\in H^0(K_C)$, with local
  expression $\omega=h(z)dz$ on $U$, it holds that
  \begin{gather*}
    dj_x(\xi_u)(\omega)= \xi_u \cup \omega=\left[\phi \cdot
      \omega\right ]=\left[\frac{\delbar (bh)}{z}\right].
  \end{gather*}
  If $f_u$ is an elementary potential, the functions
  \begin{gather*}
    \frac{b}{z} + f_u \quad \text{and} \quad b \cd \frac{h-h(0)}{z}
  \end{gather*}
  are smooth on $C$. Hence
  \begin{gather*}
    \delbar\left(\frac{bh}{z}\right) = \debar \left ( b\cd \frac{h -
        h(0)}{z} \right ) + h(0) \cd \debar \left ( \frac{b}{z} + f_u
    \right ) - h(0) \debar f_u.
  \end{gather*}
  Thus $dj_x(\xi_u)(\omega)=
  h(0)\cd [-\delbar f_u].$ As usual we identify $H^{0,1}(C)$ with the
  space of antiholomorphic forms.  Thus since $h(0) = \om(u)$, using
  Lemma \ref{berg-potelem} and the fact that $\ove{k}_u$ is
  antiholomorphic, we get \eqref{dj}.
  
  We also recall (see \cite[Lemma 2.3]{cfg}) that for
  $\beta\in H^0(C, 2K_C)=H^1(C, T_C)^*$, we have
  $\beta(\xi_u)=2\pi i \beta(u)$.  It follows that for
  $\Phi \in H^0(Z,2K_Z) $
  \begin{gather*}
    \Phi((u,0),(0,\bar{v})) = - \frac{1}{4 \pi^2} \Phi ( \xi_u \otimes \xi_{\bar{v}}).
  \end{gather*}

  Now we can prove the statement.  Without loss of generality we can
  assume that $\Omega=p^*\omega\wedge q^*\overline{\omega'}$, with
  $\omega, \omega'\in H^0(K_C)$.
%
Then
  \begin{gather*}
    (  \bra^*\Omega)_{(x,y)}((u,0), (0, \overline{v}))=-\frac{1}{4\pi ^2}\bra^*(p^*\omega\wedge q^*\overline{\omega'})(\xi_u\otimes \xi_{\overline{v}})=\\
    =-\frac{1}{4\pi^2}(p^*\omega\wedge q^*\overline{\omega'})(\bra(\xi_u\otimes \xi_{\overline{v}})).
  \end{gather*}
 It follows from \eqref{dj} that
  \begin{gather*}
    \xi_{\overline{v}} \cup ( \xi_u \cup \om) =-2\pi\omega(u)\cd
    \xi_{\overline{v}} \cup \ove{k}_u= -2\pi\omega(u)\cd
    \overline{\xi_v \cup k_u} =
    \\
    =4\pi^2 \omega(u) \cd \ove{k_u (v)}\cd k_v.
  \end{gather*}
  Using this and  Lemma  \ref{inversi}  we get
  \begin{gather*}
        (  \bra^*\Omega)_{(x,y)}((u,0), (0, \overline{v}))=
 -   \frac{1}{4\pi^2}  Q^*(\overline{\omega'}, \bra(\xi_u\otimes \xi_{\overline{v}})\omega)=\\
    =-\frac{1}{4\pi^2} Q^*(\overline{\omega'},\xi_{\overline{v}}
    \cup ( \xi_u \cup \om)) =- \omega(u) \cd \ove{k_u (v)}\cd
    Q^*(\ove{\om'}, k_v).
  \end{gather*}
  Since $ Q^*(\ove{\om'}, k_v)= i\cd \ove{\om'(v)}$ and using 
 Lemma\ref{berg} we finally get
  \begin{gather*}
    ( \bra^*\Omega)_{(x,y)}((u,0), (0, \overline{v})) =- i\omega(u)
    \ove{\om'(v)} \cd \ove{k_u (v)} =
    \\
    = - i \cd (\Omega \cdot \Berg)_{(x,y)} ((u, 0), (0, \overline{v}))
  \end{gather*}
\end{proof}

      \section{The form $\hat{\eta}$}

      Fix a smooth complex projective curve $C$ of genus $g >0$ and
      let $\Delta\subset S=C\times C $ be the diagonal. In this
      section we recall the definition of the meromorphic form
      $\hat{\eta}\in H^0(C, K_S(2\Delta))$ constructed in
      \cite{cpt,cfg}, which governs the second fundamental form of the
      Torelli map with respect to the Siegel metric. Next we recall
      from \cite{bcfp} the analysis of its cohomology class. Finally
      we prove our second main result, i.e. Theorem B.

      \subsection{} The construction of the form $\hat{\eta}$ goes as
      follows.  For $x\in C$, let
$$j_x: H^0(C, K_C(2x))\hookrightarrow H^1(C\setminus\{x\}, \mathbb{C})=H^1(C, \mathbb{C})$$
be the map that associates to $\omega\in H^0(C, K_C(2x))$ its de Rham
cohomology class. This map is an injection since $C\neq \PP^1$. As
$H^{1,0}(C)\subset j_x(H^0(C, K_C(2x)))$ and $h^0(C, K_C(2x))=g+1$,
the preimage $j_x^{-1}(H^{0,1}(C))$ is a line. Thus, fixed a chart
$(U,z)$ centered at $x\in C$, there exists a unique element $\phi$ in
this line such that on $U\setminus\{x\}$
$$\phi=\left(\frac{1}{z^2}+h(z)\right)dz$$
with $h\in \mathcal{O}_C(U)$. Set $u=\frac{\partial}{\partial z}(x),$ and define the map
$$\eta_x: T^{1,0}_xC\rightarrow H^0(C, K_C(2x)), \quad \lambda u\mapsto \eta_x(\lambda u):=\lambda\phi.$$
It is easy to see that $\eta_x$ does not depend on the choice of the
local coordinate.  In the following we will also use the fact that if
$f_u$ is an elementary potential, then $\del f_u=\eta_u$, see
\cite[Lemma 3.1]{cfg}.

Next consider the line bundle $L:=K_S(2\Delta)$ on $S$ and set
$$V:=p_*(q^*K_C(2\Delta)), \quad E:=p_*L.$$
By the projection formula $E=K_C\otimes V$. Also, since $q^*K_C(2\Delta)|_{\{x\}\times C}=q^*K_C(2x)$, we have that $H^0(p^{-1}(x), q^*K_C(2\Delta))\simeq H^0(C, K_C(2x))$ and the fiber of the holomorphic vector bundle $V\rightarrow C$ on $x\in C$ is isomorphic to $H^0(C, K_C(2x))$. Thus $\eta_x\in E_x$.  More precisely, the map $x\mapsto \eta_x$ is a holomorphic section of $E$ (\cite[Proposition 3.4]{cfg}).

Finally, since $E=p_*L$, there is an isomorphism between $H^0(C, E)$ and $H^0(S, L)$ that associates to $\alpha\in H^0(C, E)$ the section $\hat{\alpha}$ of $L$ such that $\alpha_x=\hat{\alpha}|_{\{x\}\times C}\in E_x$. The form $\hat{\eta}\in H^0(S, K_S(2\Delta))$ is defined as the 
holomorphic section of $L$ corresponding to $\eta\in H^0(C, E)$.
Note that, in particular, for $u\in T^{1,0}_xC$ and $v\in T^{1,0}_xC$ with $x\neq y$, it holds
\begin{gather}
  \label{etaeta}
  \eta_x(u)(v)=\hat{\eta}(u,v).
\end{gather}

\subsection{}
The form $\hat{\eta}$ also appears in an unpublished book of Gunning
\cite{gunning} under the name of \emph{intrinsic double differential
  of the second kind}.

\subsection{}
The importance of $\hat{\eta}$ comes from the fact that the second fundamental form of the Torelli map outside the hyperelliptic locus coincides with the multiplication by $\hat{\eta}$ \cite {cpt,cfg}.
The form $\hat{\eta}$ has been further studied in \cite{bcfp} in relation with projective structures on compact Riemann surfaces.
Moreover   Section $5$ in \cite{bcfp} contains an  analysis of the cohomology class of the form $\hat{\eta}$. Denoting $j: S\setminus\Delta\hookrightarrow S$ the inclusion map, it follows from the exact sequence of homology groups for the pair $(S, \Delta)$ and Poincaré and Lefschetz dualities, that the homomorphism $j^*: H^2(S, \mathbb{Z})\rightarrow H^2(S \setminus \Delta, \mathbb{Z})$ is surjective and its kernel is generated by the (pure) class of the diagonal. Consequently, the Hodge decomposition of $H^2(S)$ induces a decomposition of  $H^2(S \setminus \Delta)$. In particular,  for any $\zeta\in H^0(C, K_S(2\Delta))$
there is $[\gamma]\in H^2(S)$ such that
$[\zeta]=j^*[\gamma] \in H^2(S\setminus \Delta)$.
Moreover  the $(0,2)$ part of $[\gamma]$ vanishes. So
$$[\gamma]=\gamma^{2,0}+\gamma^{1,1}$$
where $\gamma^{2,0}$ is holomorphic and $\gamma^{1,1}$ is harmonic of
type $(1,1)$.  (Harmonicity is for any product metric.) In this
context, the fundamental result on the cohomology of $\hat{\eta}$ is
that $$[\hat{\eta}]=j^*[\gamma^{1,1}]\in H^2(S\setminus \Delta),$$
where $\gamma^{1,1}$ is a harmonic ${(1,1)-}$form on $S$. That is,
$\hat{\eta}$ has cohomology class in $H^2(S\setminus \Delta)$ of pure
type $(1,1)$. In fact, this implies the characterization of
$\hat{\eta}$ as the unique element (up to multiples) of
$H^0(S, K_S(2\Delta))$ with cohomology class in
$H^2(S\setminus \Delta)$ of pure type $(1,1)$. In the following we
give an explicit description of the harmonic representative
$\gamma^{1,1}$ of $\hat{\eta}$.


\subsection{} Denote by $\mathcal{A}^{p,q}$ the sheaf of smooth
differential forms of type $(p,q)$ on $S$.  Denote by
$\mathcal{A}^{p,q}(n\Delta)$ be the sheaf of $(p,q)-$forms having a
pole of order at most $n$ on $\Delta$, i.e. those forms $\om$ such
that $x^n \om$ is smooth of type $(p,q)$, where $x=0$ is a local
equation of $\Delta$.


For $u\in(T_xC)_{\mathbb{C}}$, denote by $u^{1,0}$ the
$(1,0)-$component of $u$ and set $f_u:=f_{u^{1,0}}$.  For
$u\in(T_xC)_{\mathbb{C}}$ and $v\in (T_yC)_{\mathbb{C}}$, set
  \begin{gather*}
    \alpha(u,v):=2f_v(p)+f_u(q)
  \end{gather*}
  It is clear that 
  $\alpha\in H^0(S, \mathcal{A}^{1,0}(\Delta))$.

We can now prove Theorem B.
\begin{teo}
\label{teoeta}
  The Bergman kernel is the $(1,1)-$harmonic representative
  of the cohomology class of $\hat{\eta}\in H^0(S, K_S(2\Delta))$ in
  $H^2(S\setminus \Delta)$. More precisely,
	$$\hat{\eta}- {2\pi}\Berg  =d\alpha, $$
        that is $\del\alpha=\hat{\eta}$ and
        $\delbar \alpha=- 2\pi \Berg$.
      \end{teo}
      \begin{proof}
        We first observe that for $u\in(T_xC)_{\mathbb{C}}$ and
        $v\in (T_yC)_{\mathbb{C}}$, denoted by $U, V$ two vector
        fields on $C$ such that $U_{x}=u$ and $V_y=v$, since
        $[(U,0), (0,V) ]=0$, we have that
$$d	\alpha((u,0), (0, v))=(U,0)(\alpha(0,v))-(0,V)(\alpha(u,0)).$$
Now assume $u\in T^{1,0}_xC$, $v\in T^{1,0}_yC$ and that $U$ and $V$ are (1,0). For $x\neq y$,
$\del f_u=\eta_x(u)$ and we get
$$\del\alpha((u,0), (0, v))=2(U,0)f_v-(0,V)f_u=$$
$$=2\del f_v(u)-\del f_u(v)=2\eta_y(v)(u)-\eta_x(u)(v)=\hat{\eta}((u,0), (0, v)),$$
where for the last equality we used \eqref{etaeta} and  the symmetry of  $\hat{\eta}$  (see \cite[Lemma 3.5]{cfg}).
Similarly using  Lemma \ref{berg-potelem} 
\begin{gather*}
\delbar\alpha((u,0), (0, \overline{v}))=-(0,\ove{V})f_u=-\delbar f_u(\overline{v})
=- {2\pi} \Berg((u,0), (0, \overline{v})).
\end{gather*}
\end{proof}

\end{document}